\documentclass[a4paper,10pt]{amsart}
\usepackage{amsmath,amsfonts,amssymb, amsthm}
\usepackage{epsfig,psfrag}
\usepackage[all]{xy}

\def\C{\mathbb{C}}

\def\SH{\mathbf{S}{\mathbf{H}}}
\def\tSH{\widetilde{\mathbf{S}\mathbf{H}}}
\def\tD{\tilde{D}}

\def\qed{\hfill $\sqcap \hskip-6.5pt \sqcup$}
\def\N{\mathbb{N}}
\def\Z{\mathbb{Z}}

\def\cb{\mathbf{c}}

\def\leqs{\leqslant}
\def\geqs{\geqslant}

\theoremstyle{plain}
\newtheorem{theo}{Theorem}[section]
\newtheorem{lem}[theo]{Lemma}
\newtheorem{prop}[theo]{Proposition}

\newtheorem*{theoab}{Theorem}

\theoremstyle{definition}

\theoremstyle{remark}

\numberwithin{equation}{section}

\title{A presentation of the deformed $W_{1+\infty}$ algebra}
\author{N. Arbesfeld and O. Schiffmann}

\begin{document}

\begin{abstract}
We provide a generators and relation description of the deformed $W_{1+\infty}$-algebra introduced
in previous joint work of E. Vasserot and the second author. This gives a presentation of the (spherical)
cohomological Hall algebra of the one-loop quiver, or alternatively of the spherical degenerate
double affine Hecke algebra of $GL(\infty)$.
\end{abstract}

\maketitle

\centerline{\small{INTRODUCTION}}

\vspace{.1in}

In the course of their work on the cohomology of the moduli space of $U(r)$-instantons on $\mathbb{P}^2$ in relation to $W$-algebras and the AGT conjecture (see \cite{SV-AGT}) E. Vasserot and the second author introduced a certain one-parameter deformation $\SH^{\cb}$ of the enveloping algebra of the Lie algebra $W_{1+\infty}$ of algebraic differential operators on $\mathbb{C}^*$. The algebra $\SH^{\cb}$ --which is defined in terms of Cherednik's double affine Hecke algebras-- acts on the above mentioned cohomology spaces (with a central character depending on the rank $n$ of the instanton space). For the same value of the central character, $\SH^{\cb}$ is also strongly related to the affine $W$ algebra of type $\mathfrak{gl}_n$, and has the same representation theory
(of admissible modules) as the latter. The same algebra $\SH^{\cb}$ arises again as the (spherical) cohomological Hall algebra of the quiver with one vertex and one loop, and as a degeneration of the (spherical) elliptic Hall algebra (see \cite[Sec. 4, 8]{SV-AGT}. 
It also independently appears in the work of Maulik and Okounkov on the AGT conjecture, see \cite{OM}. 

\vspace{.1in}

The definition of $\SH^{\cb}$ given in \cite{SV-AGT} is in terms of a stable limit of spherical degenerate double
affine Hecke algebras, and does not yield a presentation by generators and relations. In this note, we provide such a
presentation, which bears some resemblance with Drinfeld's new realization of
quantum affine algebras and Yangians. Namely, we show that $\SH^{\cb}$ is generated by families of 
elements in degrees $-1, 0, 1$, modulo some simple quadratic and cubic relations (see Theorems~\ref{T:1}, \ref{T:2}).
 
\vspace{.1in}

The definition of $\SH^{\cb}$ is recalled in Section 1. In the short Section 2 we briefly recall the links between $\SH^{\cb}$ and Cherednik algebras, resp. W-algebras. The presentation of $\SH^{\cb}$ is given in Section 3, and
proved in Section 4. 
Although we have tried to make this note as self-contained as possible, there are multiple references to statements in \cite{SV-AGT} and the reader is advised to consult that paper (especially Sections 1 and 8) for details.

\vspace{.2in}

\section{Definition of $\SH^{\cb}$}

\vspace{.1in}

\paragraph{\textbf{1.1. Symmetric functions and Sekiguchi operators.}} Let $\kappa$ be a formal parameter, and let us set $F=\C(\kappa)$. Let us denote by $\Lambda_F$ the ring
of symmetric polynomials in infinitely many variables with coefficients in $F$, i.e.
$$\Lambda_F=F[X_1,X_2, \ldots ]^{\mathfrak{S}_{\infty}}=F[p_1, p_2, \ldots].$$
For $\lambda$ a partition, we denote by $J_{\lambda}$ the integral form of the Jack polynomial associated to $\lambda$ and to the parameter $\alpha=1/\kappa$.
It is well-known that $\{J_{\lambda}\}$ forms a basis of $\Lambda_F$ (see e.g. \cite{Stanley}, or \cite[Sec. 1.3, 1.6]{SV-AGT}).

The polynomials $J_{\lambda}$ arise as the joint spectrum of a family of commuting differential operators $\{D_{0,l}\}, l \geq 1$ called Sekiguchi operators. In the above normalization, these
may be characterized through the following relations :

\begin{equation}\label{E:sekiguchi1}
D_{0,l} (J_{\lambda})=\sum_{s \in \lambda} c(s)^{l-1} J_{\lambda}
\end{equation}
where $s$ runs through the set of boxes in the partition $\lambda$, and where $c(s)=x(s)-\kappa y(s)$ is the content of $s$. Here $x(s), y(s)$ denote the $x$ and $y$-coordinates
of the box $s$, when $\lambda$ is drawn according to the continental convention (see \cite[Sec. 0.1]{SV-AGT}).

We denote by $D_{l,0} \in \text{End}(\Lambda_F)$ the operator of multiplication by the power-sum function $p_l$.
 
\vspace{.1in}

\paragraph{\textbf{1.2. The algebras $\SH^+$ and $\SH^>$.}} Let $\SH^+$ be the unital subalgebra of $\text{End}(\Lambda_F)$ generated by $\{D_{0,l}, D_{l,0}\;|\; l \geqslant 1\}$.
For $l \geq 1$ we set $D_{1,l}=[D_{0,l+1}, D_{1,0}]$. This relation is still valid when $l=0$, and we furthermore have 
\begin{equation}\label{E:sekiguchi2}
[D_{0,l}, D_{1,k}]=D_{1,k+l-1} \qquad l \geqslant 1, k \geqslant 0.
\end{equation}

We denote by $\SH^>$ the unital subalgebra of $\SH^+$ generated by $\{D_{1,l}\;|\; l \geqslant 0\}$, and by $\SH^0$ the unital subalgebra of $\SH^+$
generated by the Sekiguchi operators $\{D_{0,l}\;|\; l \geqslant 1\}$. It is known (and easy to check from (\ref{E:sekiguchi1})) that the $D_{0,l}$ are algebraically independent, i.e.
$\SH^0=F[D_{0,1}, D_{0,2}, \ldots]$.

Observe that by (\ref{E:sekiguchi2}), the operators $ad(D_{0,l})$ preserve the subalgebra $\SH^>$. This allows us to view $\SH^+$ as a semi-direct product of $\SH^0$ and $\SH^>$.
In fact, the multiplication map induces an isomorphism 
\begin{equation}\label{E:triangSH}
\SH^> \otimes \SH^0 \simeq \SH^+
\end{equation}
(see \cite[Prop. 1.18]{SV-AGT}). 

\vspace{.1in}

\paragraph{\textbf{1.3. Grading and filtration}} The algebra $\SH^+$ carries an $\N$-grading, defined by setting
$D_{0,l}, D_{1,k}$ in degrees zero and one respectively. This grading, which corresponds to the degrees as operators on polynomials will be called the \textit{rank} grading.
It also carries an $\N$-filtration compatible with the rank grading, induced from the filtration by the order of differential operators. It may be characterized as follows, see \cite[Prop. 1.2]{SV-AGT}~: $\SH^+[\leqs d]$ is the space of elements $u \in \SH^>$ satisfying
$$ad (z_1) \circ \cdots \circ ad(z_{d+1}) (u)=0$$
for all $z_1, \ldots, z_{d+1} \in F[D_{1,0}, D_{2,0}, \ldots]$. We have $\SH^>[\leqs 0]=F[D_{1,0}, D_{2,0}, \ldots]$.
The following is proved in \cite[Lemma~1.21]{SV-AGT}. Set $D_{r,d}=[D_{0,d+1}, D_{r,0}]$ for $r \geq 1, d \geq 0$.

\vspace{.1in}

\begin{prop} (i) The associated graded algebra $gr\SH^+$ is equal to the free commutative polynomial algebra
in the generators $D_{r,d} \in gr\SH^+[r,d]$, for $r \geqs 0, d \geqs 0, (r,d) \neq (0,0)$.\\
(ii) The associated graded algebra $gr\SH^>$ is equal to the free commutative polynomial algebra
in the generators $D_{r,d} \in gr\SH^+[r,d]$, for $r \geqs 1, d \geqs 0$.
\end{prop}

\vspace{.1in}

We will need the following slight variant of the above result, which can easily be deduced from \cite[Prop. 1.38]{SV-AGT}. For $r \geqs 1$, set $D'_{r,d}=ad (D_{0,2})^d ( D_{r,0}).$ Then 
\begin{equation}\label{E:relDD'}
D'_{r,d} \in r^{d-1}D_{r,d} \oplus \SH^>[r, \leqs d-1].
\end{equation}
In particular, $gr\SH^>$ is also freely generated by the elements $D'_{r,d} \in gr\SH^>[r,d]$.

\vspace{.1in}

\paragraph{\textbf{1.4. The algebra $\SH^{\cb}$.}} Let $\SH^<$ be the opposite algebra of $\SH^>$. We denote the
generator of $\SH^>$ corresponding to $D_{1,l}$ by $D_{-1,l}$. The algebra $\SH^{\cb}$ is generated by
$\SH^>, \SH^0, \SH^<$ together with a family of central elements
$\cb=(c_0, c_1, \ldots)$ indexed by $\N$, modulo a certain set of relations involving the commutators $[D_{-1,k}, D_{1,l}]$ (see \cite[Sec. 1. 8]{SV-AGT}). In order to write down these relations, we need a few notations. Set
$\xi=1-\kappa$ and
$$G_0(s)=-\text{log}(s), \qquad G_l(s)=(s^{-1}-1)/l, \qquad l \geqs 1,$$
$$\varphi_l(s)=\sum_{q=1,-\xi, -\kappa}s^l(G_l(1-qs)-G_l(1+qs)), \quad l \geqs 1,$$
$$\phi_l(s)=s^lG_l(1+\xi s ).$$
We may now define $\SH^{\cb}$ as the algebra generated by $\SH^>, \SH^<, \SH^0$ and $F[c_0, c_1, \ldots]$ modulo the following relations~:
\begin{equation}\label{E:def-1}
[D_{0,l}, D_{1,k}]=D_{1,k+l-1}, \qquad [D_{-1,k}, D_{0,l}]=D_{-1,k+l-1},
\end{equation}
\begin{equation}\label{E:def0}
 [D_{-1,k}, D_{1,l}]=E_{k+l}, \qquad l, k \geqs 0,
\end{equation}
where the elements $E_{h}$ are determined through the formulas
\begin{equation}\label{E:rel-11}
1 + \xi \sum_{l \geqs 0}E_ls^{l+1}=\text{exp} \left( \sum_{l \geqs 0} (-1)^{l+1} c_l \phi_l(s)\right)\text{exp}
\left(\sum_{l \geqs 0}D_{0,l+1}\varphi_l(s)\right).
\end{equation}

\vspace{.1in}

Set $\SH^{0,\cb}=\SH^0 \otimes F[c_0, c_1, \ldots]$. One can show that the multiplication map provides an isomorphism of $F$-vector spaces 
$$\SH^> \otimes \SH^{0,\cb} \otimes \SH^< \simeq \SH^{\cb}.$$
Putting the generators
$D_{\pm 1,k}$ in degree $\pm 1$ and the generators $D_{0,l}, c_i$ in degree zero induces an $\mathbb{Z}$-grading on $\SH^{\cb}$. One can show that the order filtration on $\SH^>, \SH^<$ can be extended to a filtration on the whole
$\SH^{\cb}$, but we won't need this last fact.

\vspace{.2in}

\section{Link to W-algebras, Cherednik algebras and shuffle algebras}

\vspace{.1in}

\paragraph{\textbf{2.1. Relation the Cherednik algebras.}} Let $\omega$ be a new formal parameter and let $\SH^\omega$ be the specialization of $\SH$ at $c_0=0, c_i=-\kappa^i \omega^i$.
Let $\mathbf{H}_n$ be Cherednik's degenerate (or trigonometric) double affine Hecke algebra with parameter $\kappa$ 
(see \cite{Cherednik}).
Let $\SH_n \subset \mathbf{H}_n$ be its spherical subalgebra. The following result
shows that $\SH^\omega$ may be thought of as the stable limit of $\SH_n$ as $n$ goes to infinty (see \cite[Sec. 1.7]{SV-AGT})~:

\vspace{.1in}

\begin{theoab} For any $n$ there exists a surjective algebra homomorphism $\Phi_n~: \SH^\omega \to \SH_n$
such that $\Phi_n(\omega)=n$. Moreover
$\bigcap_n \text{Ker}\; \Phi_n =\{0\}.$ 
\end{theoab}

\vspace{.1in}

\paragraph{\textbf{2.2. Realization as a shuffle algebra.}} Consider the rational function
$$g(z)=\frac{h(z)}{z}, \qquad h(z)=(z+1-\kappa)(z-1)(z+\kappa).$$
Following \cite{Feigin}, we may associate to $g(z)$ an $\N$-graded associative $F$-algebra $A_{g(z)}$, the \textit{symmetric shuffle algebra of $g(z)$} as follows. As a vector space,
$$A_{g(z)}=\bigoplus_{n \geqs 0} F[z_1, \ldots, z_n]^{\mathfrak{S}_n}$$
with multiplication given by
$$P(z_1, \ldots, z_r) \star Q(z_1, \ldots, z_s)=\sum_{\sigma \in Sh_{r,s}} \sigma \cdot \bigg(  \hspace{-.1in} \prod_{\substack{1 \leqs i \leqslant r\\r+1 \leqs j \leqs r+s}}\hspace{-.1in} g(z_i-z_j)\cdot P(z_1, \ldots, z_r) Q(z_{r+1}, \ldots, z_{r+s})\bigg)$$
where $Sh_{r,s} \subset \mathfrak{S}_{r+s}$ is the set of $(r,s)$ shuffles inside the symmetric group $\mathfrak{S}_{r+s}$.
Let $S_{g(z)} \subseteq A_{g(z)}$ denote the subalgebra generated by $A_{g(z)}[1]=F[z]$. The following
is proved in \cite[Cor. 6.4]{SV-AGT}~:

\vspace{.1in}

\begin{theoab} The assignment $S_{g(z)}[1] \ni z^l \mapsto D_{1,l}, \; l \geqs 0$ induces an isomorphism of $F$-algebras
$$S_{g(z)} \stackrel{\sim}{\longrightarrow} \SH^>.$$
\end{theoab}

\vspace{.1in}

\paragraph{\textbf{2.3. Relation to $W$-algebras.}} Let $W_{1+\infty}$ be the universal central extension of the Lie algebra of all differential operators on $\mathbb{C}^*$ (see e.g. \cite{Kac}). This is a $\mathbb{Z}$-graded and $\N$-filtered Lie algebra. The following result shows that $\SH$ may be thought of as a deformation of the universal 
enveloping algebra $U(W_{1+\infty})$ of $W_{1+\infty}$ (see \cite[App. F]{SV-AGT})~:

\vspace{.1in}

\begin{theoab} The specialization of $\SH^{\cb}$ at $\kappa=1$ and $c_i=0$ for $i \geqs 1$ is isomorphic to $U(W_{1+\infty})$.
\end{theoab}

\vspace{.1in}

More interesting is the fact that, for certain good choices of the parameters $c_0, c_1, \ldots$, a suitable completion
of $\SH^{\cb}$ is isomorphic to the current algebra of the (affine) $W$-algebra $W(\mathfrak{gl}_r)$ (see e.g. \cite[Sec. 3.11]{Arakawa}). Fix an integer $r \geqs 1$, $k \in \mathbb{C}$ and let $(\varepsilon_1, \ldots, \varepsilon_r)$ be new formal parameters. Let $\mathfrak{U}(W_k(\mathfrak{gl}_r))'$ be the formal current algebra of $W(\mathfrak{gl}_r)$ at level $k$, defined over the field $F(\varepsilon_1, \ldots, \varepsilon_r)$ (see \cite[Sec. 8.4]{SV-AGT} for details). Let $\SH^{(r)}$ be the specialization of $\SH^{\cb}$ to $\kappa=k+r$, $c_i=\varepsilon_1^i + \cdots + \varepsilon_r^i$ for $i \geqs 0$. The following is proved in \cite[Cor. 8.24]{SV-AGT}, to which we refer for details.

\vspace{.1in}

\begin{theoab} There is an embedding $\SH^{(r)} \to \mathfrak{U}(W_k(\mathfrak{gl}_r))'$ with a dense image,
which induces an equivalence between the category of admissible $\SH^{(r)}$-modules and the category
of admissible $\mathfrak{U}(W_k(\mathfrak{gl}_r))'$-modules.
\end{theoab}

\vspace{.2in}

\section{Presentation of $\SH^{+}$ and $\SH^{\cb}$}

\vspace{.1in}

\paragraph{\textbf{3.1. Generators and relations for $\SH^+$.}} Consider the $F$-algebra $\tSH^{+}$ generated by elements $\{\tD_{0,l}\;|\; l \geqslant 1\}$ and $\{\tD_{1,k}\;|\; k \geqslant 0\}$
subject to the following set of relations~:
\begin{equation}\label{E:def1}
 [\tD_{0,l}, \tD_{0,k}]=0, \qquad \forall\; l, k \geqslant 1,
\end{equation}
\begin{equation}\label{E:def2}
[\tD_{0,l}, \tD_{1,k}]=\tD_{1,l+k-1}, \qquad \forall\; l \geqslant 1, k \geqslant 0, 
\end{equation}
\begin{equation} \label{E:def3}
\big( 3 [\tD_{1,2}, \tD_{1,1}] - [\tD_{1,3}, \tD_{1,0}] + [\tD_{1,1}, \tD_{1,0}]\big) + \kappa(\kappa-1) ( \tD_{1,0}^2 + [\tD_{1,1}, \tD_{1,0}])=0
\end{equation}
\begin{equation}\label{E:def4}
 [\tD_{1,0}, [\tD_{1,0}, \tD_{1,1}]]=0.
\end{equation}
Let $\tSH^0=F[\tD_{0,1}, \tD_{0,2}, \ldots]$ denote the subalgebra of $\tSH^+$ generated by $\tD_{0,l}, l \geqslant 1$, and let $\tSH^>$ be the subalgebra generated
by $\tD_{1,k}, k \geqslant 0$. The algebras $\tSH^+, \tSH^0, \tSH^>$ are all $\N$-graded, where $\tD_{0,l}$ and $\tD_{1,k}$ are placed in degrees zero and one respectively.
According to the terminology used for $\SH^+$, we call this grading the \textit{rank grading}.

\vspace{.1in}

\begin{theo}\label{T:1} The assignment $\tD_{0,l} \mapsto D_{0,l}, \tD_{1,k} \mapsto D_{1,k}$ for $l \geqslant 1, k \geqslant 0$ induces an isomorphism
of graded $F$-algebras 
$$\phi~: \tSH^+ \stackrel{\sim}{\longrightarrow} \SH^+.$$
\end{theo}
Obviously, the map $\phi$ restricts to isomorphisms $\tSH^0 \simeq \SH^0, \tSH^> \simeq \SH^>$. Note however that $\tSH^>$ is \textit{not}
generated by the elements $\tD_{1,k}$ with the sole relations (\ref{E:def3}, \ref{E:def4}). Theorem~\ref{T:1} is proved in Section 4.

\vspace{.1in}

\paragraph{\textbf{3.2. Generators and relations for $\SH^{\cb}$.}} For the reader's convenience, we write down the presentation of $\SH^{\cb}$, an immediate corollary
of Theorem~\ref{T:1} above. Let $\tSH^{\cb}$ be the algebra generated by elements $\{\tD_{0,l}\;|\; l \geqslant 1\}$, $\{\tD_{\pm 1, k}\;|\; k \geqslant 0\}$
and $\{\tilde{\cb}_i\;|\; i \geqslant 0\}$ subject to the following set of relations~:

\begin{equation}\label{E:ndef1}
 [\tD_{0,l}, \tD_{0,k}]=0, \qquad \forall\; l, k \geqslant 1,
\end{equation}
\begin{equation}\label{E:ndef2}
[\tD_{0,l}, \tD_{1,k}]=\tD_{1,l+k-1}, \qquad [\tD_{-1,k}, \tD_{0,l}]=\tD_{-1,l+k-1} \qquad \forall\; l \geqslant 1, k \geqslant 0, 
\end{equation}
\begin{equation} \label{E:ndef3}
\big( 3 [\tD_{1,2}, \tD_{1,1}] - [\tD_{1,3}, \tD_{1,0}] + [\tD_{1,1}, \tD_{1,0}]\big) + \kappa(\kappa-1) ( \tD_{1,0}^2 + [\tD_{1,1}, \tD_{1,0}])=0
\end{equation}
\begin{equation} \label{E:nd2ef3}
\big( 3 [\tD_{-1,2}, \tD_{-1,1}] - [\tD_{-1,3}, \tD_{-1,0}] + [\tD_{-1,1}, \tD_{-1,0}]\big) + \kappa(\kappa-1) ( -\tD_{1,0}^2 + [\tD_{-1,1}, \tD_{-1,0}])=0
\end{equation}
\begin{equation}\label{E:ndef4}
 [\tD_{1,0}, [\tD_{1,0}, \tD_{1,1}]]=0, \qquad [\tD_{-1,0},[\tD_{-1,0}, \tD_{-1,1}]]=0,
\end{equation}
\begin{equation}\label{E:ndef0}
 [\tD_{-1,k}, \tD_{1,l}]=\tilde{E}_{k+l}, \qquad l, k \geqs 0,
\end{equation}
where the $\tilde{E}_{l}$ are defined by the formula (\ref{E:rel-11}).

\vspace{.1in}

\begin{theo}\label{T:2} The assignment $\tD_{0,l} \mapsto D_{0,l}, \tD_{\pm 1,k} \mapsto D_{\pm 1,k}$ for $l \geqslant 1, k \geqslant 0$ and $\tilde{\cb}_i \mapsto \cb_i$ for $i \geqslant 0$
induces an isomorphism
of $F$-algebras 
$$\phi~: \tSH^{\cb} \stackrel{\sim}{\longrightarrow} \SH^{\cb}.$$
\end{theo}

Coupled with the Theorems in Section~2.3., this provides a potential 'generators and relations' approach to the study of the category of admissible modules over the W-algebras $W_k(\mathfrak{gl}_r)$.

\vspace{.2in}

\section{Proof of Theorem~\ref{T:1}}

\vspace{.1in}

\paragraph{\textbf{4.1.}} Let us first observe that $\phi$ is a well-defined algebra map, i.e. that relations (\ref{E:def1}--\ref{E:def4}) hold in $\SH^+$.
For (\ref{E:def1}, \ref{E:def2}) this follows from the definition of $\SH^+$ and \cite[(1.38)]{SV-AGT}. Equation (\ref{E:def3}) may be checked directly, e.g.
from the Pieri rules (see \cite[(1.26)]{SV-AGT}), or from the shuffle realization of $\SH^>$ (see \textbf{4.2.} below). As for equation (\ref{E:def4}), we have
by \cite[(1.35)]{SV-AGT}, $[[D_{1,1}, D_{1,0}],D_{1,0}]=[D_{2,0}, D_{1,0}]=0$. The map $\phi$ is surjective by construction; in the rest of the
proof, we show that it is injective as well.

\vspace{.1in}

\paragraph{\textbf{4.2.}} Using relation (\ref{E:def2}) it is easy to see that any monomial in the generators $\tD_{0,l}, \tD_{1,k}$ may be expressed
as a linear combination of similar monomials, in which all $\tD_{0,l}$ appear on the right of all $\tD_{1,k}$. Hence the multiplication
map $\tSH^> \otimes \tSH^0 \to \tSH^+$ is surjective. Since $\phi$ clearly restricts to an isomorphism $\tSH^0 \simeq \SH^0$ we only have to
show, by (\ref{E:triangSH}), that $\phi$ restricts to an isomorphism $\tSH^> \simeq \SH^>$. Our strategy will be to construct a suitable filtration
on $\tSH^>$ mimicking the order filtration of $\SH^>$ and to pass to the associated graded algebras.

\vspace{.1in}

\paragraph{\textbf{4.3.}} We begin by proving directly, using the shuffle realization of $\SH^>$, that $\phi$ is an isomorphism in ranks one and two.
This is obvious in rank one since $\phi$ is a graded map and the only relation in rank one is (\ref{E:def2}).

Suppose $\sum \alpha_i D_{1,k_i}D_{1,l_i}=0$ is a relation in rank two. The shuffle realization then implies $\sum \alpha_i z^{k_i}\star z^{l_i}=0$ so that $$h(z_1-z_2)\big(\sum \alpha_i z^{k_i}_1z^{l_i}_2\big)=h(z_2-z_1)\big(\sum\alpha_i z_1^{l_i}z_2^{k_i}\big).$$ Therefore $\sum \alpha_i z^{k_i}_1z^{l_i}_2=h(z_2-z_1)P(z_1,z_2)$ where $P(z_1,z_2)$ is
some symmetric polynomial in $z_1, z_2$. Hence $\sum \alpha_i z^{k_i}_1z^{l_i}_2$ is a linear combination of polynomials of the form $h(z_2-z_1)(z_1^kz_2^l+z_1^lz_2^lk)$ so that $\sum \alpha_i D_{1,k_i}D_{1,l_i}$ is a linear combination of expressions of the form 
\begin{equation}\label{E:rank2}
\begin{split}
3[D_{1,l+2},D_{1,k+1}]-&3[D_{1,l+1},D_{1,k+2}]-[D_{1,l+3},D_{1,k}]+[D_{1,l},D_{1,k+3}]+[D_{1,l+1},D_{1,k}]-[D_{1,l},D_{1,k+1}]
\\&+ \kappa(\kappa-1)(D_{1,k}D_{1,l}+D_{1,l}D_{1,k}+[D_{1,l+1},D_{1,k}]-[D_{1,l},D_{1,k+1}].\end{split}
\end{equation}
If $I$ denotes the image of (\ref{E:def3}) under the action of $F[ad \tD_{0,2}, ad \tD_{0,3}, \ldots]$ then using
(\ref{E:def2}) we see that each such expression lies in $\phi(I)$ so that $\phi$ is indeed an isomorphism in rank two.

\vspace{.1in}

We remark that the relations (\ref{E:rank2}) may be written in a more standard way using the generating functions
$D(z)=\sum_{l} D_{1,l}z^{-l}$ as follows :
\begin{equation}\label{E:def3'}
k(z-w)D(z)D(w)=-k(w-z)D(w)D(z)
\end{equation}
where $k(u)=(u-1+\kappa)(u+1)(u-\kappa)=-h(-u)$.
In particular, the defining relation (\ref{E:def3}) may be replaced by the above (\ref{E:def3'}), of which it is a special
case.

\vspace{.1in}

\paragraph{\textbf{4.4.}} We now turn to the definition of the analog, on $\tSH^>$, of the order filtration on $\SH^>$. We will proceed
by induction on the rank $r$. For $r=1, d \geqs 0$, we set
$$\tSH^>[1, \leqslant d]=\bigoplus_{k \leqs d} F \tD_{1,k}.$$
Assuming that $\tSH^>[r', \leqs d']$ has been defined for all $r'<r$ we let $\tSH^>[r, \leqs d]$ be the subspace spanned by all products 
$$\tSH^>[r', \leqs d']\cdot \tSH^>[r'', \leqs d''], \qquad r'+r''=r, d' + d''=d$$
and by the spaces
$$ad(\tD_{1,l}) \big( \tSH^>[r-1, \leqs d -l +1]\big), \qquad l=0, \ldots, d+1.$$
From the above definition, it is clear that $\tSH^>$ is a $\Z$-filtered algebra. Note that it is not obvious at the moment that
$\tSH^>[r, \leqs d] =\{0\}$ for $d <0$. Because the associated graded $gr\SH^>$ is commutative, it follows
by induction on the rank $r$ that $\phi : \tSH^> \to \SH^>$ is a morphism of filtered algebras. We denote by $gr\tSH^>$ the associated graded
of $\tSH^>$ and we let $\overline{\phi} : gr\tSH^> \to gr\SH$ be the induced map. The map $\overline{\phi}$ is graded with respect to both
rank and order. Moreover $\overline{\phi}$ is an isomorphism in ranks 1 and 2 (indeed, that the filtration as defined above coincides with the order 
filtration in rank 2 can be seen directly from \cite[(1.84)]{SV-AGT}). The rest of the proof of Theorem~\ref{T:1} consists
in checking that $\overline{\phi}$ is an isomorphism. Once more, we will argue by induction. So in the remainder of the proof, \textit{we fix
an integer $r \geqs 3$ and assume that $\overline{\phi}$ is an isomorphism in ranks} $r' <r$. 

\vspace{.1in}

\paragraph{\textbf{4.5.}} By our assumption above,  the algebra $gr\tSH^>$
is commutative in ranks less than $r$, that is $ab=ba$ whenever $rank(a) + rank (b) <r$. Our first task is
to extend this property to the rank $r$.

\vspace{.1in}

\begin{lem}\label{L:1} The algebra $gr\tSH^>$ is commutative in rank $r$.
\end{lem}
\begin{proof} We have to show that for $a \in \tSH^>[r_1, \leqs d_1], b \in \tSH^>[r_2, \leqs d_2]$ and
$r_1+r_2=r$ we have
\begin{equation}\label{E:proof1}
[a,b] \in \tSH^>[r, \leqs d_1+d_2-1].
\end{equation}
We argue by induction on $r_1$. If $r_1=1$ then
(\ref{E:proof1}) holds by definition of the filtration. Now let $r_1 >1$ and let us further assume that
(\ref{E:proof1}) is valid for all $r_1', r_2'$ with $r_1'+r_2'=r$ and $r_1'<r_1$. We will now prove (\ref{E:proof1})
for $r_1,r_2$, thereby completing the induction step. According to the definition of the filtration, there
are two cases to consider~:

\vspace{.05in}

\noindent
Case 1) We have $a=a_1a_2$ with $a_1 \in \tSH^>[s', \leqs d'], a_2 \in \tSH^>[s'', \leqs d'']$ such that $s'+s''=r_1, d' + d''=d_1$. Then
$[a,b]=a_1[a_2,b]+ [a_1,b]a_2$. By our induction hypothesis on $r$, $[a_2,b] \in \tSH^>[s''+r_2, \leqs d''+d_2-1]$ hence
$a_1[a_2,b] \in \tSH^>[r, \leqs d_1+d_2-1]$. The term $[a_1,b]a_2$ is dealt with in a similar fashion.
 
\vspace{.05in}

\noindent
Case 2) We have $a=[\tD_{1,l}, a']$ with $a' \in \tSH^>[r_1-1, \leqs d_1-l+1]$. Then
$[a,b]=[[\tD_{1,l}, a'],b]=[\tD_{1,l}, [a',b]]-[a', [\tD_{1,l}, b]].$ By our induction hypothesis
on $r$, $[a',b] \in \tSH^>[r_1+r_2-1, \leqs d_1+d_2-l]$ hence $[\tD_{1,l}, [a',b]] \in \tSH^>[r, \leqs d_1+d_2-1]$.
Similarly, $[\tD_{1,l}, b] \in \tSH^>[r_2+1, \leqs d_2+l-1]$. The inclusion $[a', [\tD_{1,l}, b]] \in \tSH^>[r, \leqs d_1+d_2-1]$
now follows from the induction hypothesis on $r_1$.

\vspace{.05in}

We are done.
\end{proof}

\vspace{.1in}

\paragraph{\textbf{4.6.}} We now focus on the filtered piece of order $\leqs 0$ of $\tSH^>$. We inductively define elements $\tD_{l,0}$ for $l \geqslant 2$ by
$$\tD_{l,0}=\frac{1}{l-1} [\tD_{1,1}, \tD_{l-1,0}].$$
From \cite[(1.35)]{SV-AGT} we have $\phi(\tD_{l,0})=D_{l,0}$. Since we assume are assuming that $\overline{\phi}$ is an isomorphism
in ranks less than $r$, we have $[\tD_{l,0}, \tD_{l',0}]=0$ whenever $l+l'<r$.

\vspace{.1in}

\begin{lem}\label{L:2}
We have $[\tD_{l,0}, \tD_{l',0}]=0$ for $l + l'=r$.
\end{lem}
\begin{proof} If $r=3$ this reduces to the cubic relation (\ref{E:def4}). For $r=4$ we have to consider
\begin{equation*}
 \begin{split}
  [\tD_{3,0}, \tD_{1,0}]&= \frac{1}{2} [[ \tD_{1,1}, \tD_{2,0}], \tD_{1,0}] \\
&= \frac{1}{2} [\tD_{1,1}, [\tD_{2,0}, \tD_{1,0}]] - \frac{1}{2} [\tD_{2,0}, [\tD_{1,1}, \tD_{1,0}]]\\
&=-\frac{1}{2} [ \tD_{2,0}, \tD_{2,0}]=0.
 \end{split}
\end{equation*}

Now let us fix $l, l'$ with $l+l'=r$. We have 
\begin{equation}\label{E:proof2}
 \begin{split}
  [\tD_{l,0}, \tD_{l',0}]&=\frac{1}{l-1} [[ \tD_{1,1}, \tD_{l-1,0}], \tD_{l',0}]\\
&= \frac{1}{l-1} [\tD_{1,1}, [\tD_{l-1,0}, \tD_{l',0}]] - \frac{1}{l-1} [\tD_{l-1,0}, [\tD_{1,1}, \tD_{l',0}]]\\
&=-\frac{l'}{l-1} [\tD_{l-1,0}, \tD_{l'+1,0}].
 \end{split}
\end{equation}
If $r=2k$ is even then by repeated use of (\ref{E:proof2}) we get
$$[\tD_{l,0}, \tD_{l',0}]=c[ \tD_{k}, \tD_k]=0$$
for some constant $c$. Next, suppose that $r=2k+1$ is odd, with $k \geqs 2$. Applying $ad(\tD_{1,1})$ to
$[\tD_{k+1,0}, \tD_{k-1,0}]=0$ yields the relation
\begin{equation}\label{E:proof4}
 (k+1) [\tD_{k+2,0}, \tD_{k-1,0}]+ (k-1) [\tD_{k+1, 0}, \tD_{k,0}]=0.
\end{equation}
Similarly, applying $ad(\tD_{2,1})$ to $[\tD_{k,0}, \tD_{k-1,0}]=0$ and using the relation
$[D_{k,1}, D_{l,0}]=kl D_{l+k,0}$ in $\SH^>$ (see \cite[(1.91), (8.47)]{SV-AGT}) we obtain the relation
 \begin{equation}\label{E:proof5}
 k [\tD_{k+2,0}, \tD_{k-1,0}]+ (k-1) [\tD_{k, 0}, \tD_{k+1,0}]=0.
\end{equation}
Equations (\ref{E:proof4}) and (\ref{E:proof5}) imply that $[\tD_{k+2,0}, \tD_{k-1,0}]=[\tD_{k+1,0}, \tD_{k,0}]=0$.
The general case of $[\tD_{l,0}, \tD_{l',0}]=0$ is now deduced, as in the case $r=2k$, from repeated use of
(\ref{E:proof2}).
\end{proof}

\vspace{.1in}

Note that Lemma~\ref{L:2} above implies that $\tSH^>[r,\leqs -1]=\{0\}$.

\vspace{.1in}

\paragraph{\textbf{4.7.}} Recall that $gr\SH^>$ is a free polynomial algebra in generators in the generators $D'_{s,d}$
for $s \geqs 1, d \geqs 0$. In order to prove that $\overline{\phi}$ is an isomorphism in rank $r$, it suffices, in virtue
of Lemma~\ref{L:1}, to show that the factor space
$$U_{r,d}=gr\tSH^>[r,d]\;\;\; / \;\;\big\{\sum_{\substack{r'+r''=r\\d'+d''=d}} gr\tSH^>[r',d'] \cdot gr\tSH^>[r'',d'']\big\}$$
is one dimensional for any $d \geqs 0$. Let us set, for any $s \geqs 1, d \geqs 0$
$$\tD'_{s,d}=ad(\tD_{0,2})^d( \tD_{s,0}) \in \tSH^>[s, \leqs d].$$
We will denote by the same symbol $\tD'_{s,d}$ the corresponding element of $gr\tSH^>[s,d]$. Note that $\tD'_{s,0}=\tD_{s,0}$ 
We claim that in fact $U_{r,d}=F \tD'_{r,d}$. Observe that $\phi(\tD'_{s,d})=D'_{s,d}$ for any $s,d$, hence $\tD'_{s,d} \in U_{s,d}$ for any $s \leqs r, d \geqs 0$. 
Moreover, by our general induction hypothesis on $r$ we have $U_{s,d}=F \tD'_{s,d}$ for any $s <r$ and $d \geqs 0$. 

We will prove that $U_{r,d}=F \tD'_{r,d}$ by induction on $d$. For $d=0$, this comes from Lemma~\ref{L:2}. So fix $d >0$ and let us assume that $U_{r,l}=F \tD'_{r,l}$ for all $l < d$. By definition of the filtration on $\tSH^>$, $U_{r,d}$ is linearly spanned by the classes of the elements
$$[\tD_{1,0}, \tD'_{r-1,d+1}], [\tD_{1,1}, \tD'_{r-1, d}], \ldots, [\tD_{1, d+1}, \tD'_{r-1,0}].$$
By our induction hypothesis on $d$, the elements
$$[\tD_{1,0}, \tD'_{r-1,d}], [\tD_{1,1}, \tD'_{r-1, d-1}], \ldots, [\tD_{1, d}, \tD'_{r-1,0}]$$
all belong to $F\tD'_{r,d-1} \oplus \tSH^>[r,\leqs d-2]$. Applying $ad(\tD_{0,2})$, we see that
\begin{equation}\label{E:proof6}
[\tD_{1,0}, \tD'_{r-1, d+1}] + [\tD_{1,1}, \tD'_{r-1,d}], \ldots, [\tD_{1,d}, \tD'_{r-1,1}] + [\tD_{1,d+1}, \tD'_{r-1,0}]
\end{equation}
all belong to $F\tD'_{r,d} \oplus \tSH^>[r,\leqs d-1]$. Next, applying $ad(\tD_{0,d+2})$ to the equality $[\tD_{1,0}, \tD_{r-1,0}]=0$ yields
\begin{equation*}
[\tD_{1,0}, \tD_{r-1,d+1}] + [\tD_{1,d+1}, \tD_{r-1,0}]=0
\end{equation*}
which implies, by (\ref{E:relDD'}), that
\begin{equation}\label{E:proof7}
[\tD_{1,0}, \tD'_{r-1,d+1}] + r^d[\tD_{1,d+1}, \tD_{r-1,0}] \in [\tD_{1,0}, \tSH^>[r-1, \leqs d]] \subseteq \tSH^>[r, \leqs d-1].
\end{equation}
The collection of inclusions (\ref{E:proof6}), (\ref{E:proof7}) may be considered as a system of linear equations in $U_{r,d}$ modulo $F\tD'_{r,d}$ in the variables $[\tD_{1,0}, \tD'_{r-1, d+1}], \ldots, [\tD_{1,d+1}, \tD'_{r-1,0}]$ whose associated matrix is
$$M=\begin{pmatrix} 1 & 0 & \cdots & 0 & 1\\ 1 & 1 & \cdots & 0 & 0\\ 0 & 1 & \ddots &\vdots & \vdots \\ \vdots & \vdots & 
\ddots & 1 & 0 \\ 0 & 0 & \cdots & 1 & -r^d \end{pmatrix}$$
is invertible. We deduce that $[\tD_{1,0}, \tD'_{r-1, d+1}], \ldots, [\tD_{1,d+1}, \tD'_{r-1,0}]$ all belong to the space
$F \tD'_{r,d} \oplus \tSH^>[r, \geqs d-1]$ as wanted. This closes the induction step on $d$. We have therefore proved
that $U_{r,d}=F \tD'_{r,d}$ for all $d \geqs 0$, and hence that $\overline{\phi}$ and $\phi$ is an isomorphism in rank $r$.
This closes the induction step on $r$. Theorem~\ref{T:1} is proved. \qed

\vspace{.2in}

\centerline{\textbf{Acknowledgements}}

\vspace{.05in}

We would like to thank P. Etingof for helpful discussions. The first author would like to thank the mathematics
department of Orsay University for its hospitality and the MIT-France program for supporting his stay.

\vspace{.3in}

\small{}

\vspace{4mm}

\noindent

N. Arbesfeld, \texttt{nma@mit.edu},\\
Massachusetts Institute of Technology
Cambridge, MA 02139, USA.

O. Schiffmann, \texttt{olivier.schiffmann@math.u-psud.fr},\\
D\'epartement de Math\'ematiques, Universit\'e de Paris-Sud, B\^atiment 425
91405 Orsay Cedex, FRANCE.

\vspace{.1in}

\end{document}